\documentclass[11pt,final]{article}
\usepackage{a4}
\usepackage{amsmath}%
\usepackage{amstext}%
\usepackage{amssymb}%
\usepackage{amsthm}
\usepackage{comment}
\usepackage{listings}

\usepackage[final]{graphicx}
\usepackage{subcaption}

\usepackage{graphicx}
\usepackage{epstopdf}

\usepackage{xcolor}

\newtheorem{theorem}{Theorem}

\newtheorem{lemma}[theorem]{Lemma}

\newtheorem{rem}[theorem]{Remark}
\newenvironment{remark}{\rem\rm}{\endrem}

\newcounter{unnumber}

\newtheorem{prf}[unnumber]{Proof}
\renewenvironment{proof}{\prf\rm}{\hfill{$\blacksquare$}\endprf}
\newcommand{\R}{\mathbb{R}}%
\newcommand{\N}{\mathbb{N}}%
\newcommand{\e}{\varepsilon}%
\newcommand{\ol}{\overline}%

\newcommand{\n}{{\nabla}}

\newcommand{\To}{\longrightarrow}
\def\a{\alpha}

\def\b{\beta}
\def\e{\epsilon}
\def\d{\delta}
\def\t{\theta}
\def\g{\gamma}

\def\<{\langle}
\def\>{\rangle}
\DeclareMathOperator*\proj{proj}%
\DeclareMathOperator*\argmin{argmin}

\author{
Szil\'{a}rd Csaba L\'{a}szl\'{o} \thanks{Technical University of Cluj-Napoca, Department of Mathematics, Str. Memorandumului nr. 28, 400114 Cluj-Napoca, Romania, e-mail: slaszlo@math.utcluj.ro}}

\title{On the strong convergence of the trajectories  of a Tikhonov regularized second order dynamical system with asymptotically vanishing damping}

\begin{document}

\maketitle

\noindent \textbf{Abstract.}  This paper deals with a  second order dynamical system with vanishing damping that contains a Tikhonov regularization term, in connection  to the minimization problem of a convex Fr\'echet differentiable function $g$.
 We show that for appropriate Tikhonov regularization parameters the  value of the objective function in a generated trajectory converges fast to the global minimum of the objective function and a  trajectory  generated by the dynamical system converges weakly  to a minimizer of the objective function. We also  obtain the fast convergence of the velocities towards zero and some integral estimates. Nevertheless, our main goal is to extend and improve some recent results obtained in \cite{ABCR} and \cite{AL-nemkoz} concerning the strong convergence of the generated trajectories to an element of minimal norm from the $\argmin$ set of the objective function $g$. Our analysis also reveals that the damping coefficient and the Tikhonov regularization coefficient are strongly correlated.
\vspace{1ex}

 \vspace{1ex}


\section{Introduction}

Let $\mathcal{H}$ be a Hilbert space endowed with the scalar product $\< \cdot,\cdot\>$ and norm $\|\cdot\|$ and let
$g:\mathcal{H}\To\R$ be  a convex continuously differentiable function whose solution set $\argmin g$ is  nonempty. Assume further, that $\nabla g$ is Lipschitz continuous on  bounded sets. Consider the minimization problem
$$(P)\,\,\,\inf_{x\in\mathcal{H}}g(x)$$
in connection to the second order dynamical system
\begin{align}\label{DynSys}
\begin{cases}
\ddot{x}(t) + \frac{\a}{t^q} \dot{x}(t) + \nabla g\left(x(t)\right) +\frac{a}{t^p}x(t)=0\\
x(t_0) = u_0, \,
\dot{x}(t_0) = v_0,
\end{cases}
\end{align}
where $t_0 > 0$, $(u_0,v_0) \in \mathcal{H} \times \mathcal{H}$ and $\a,\,q,\,a,\,p>0.$

 First of all, note that the term $\frac{a}{t^p}x(t)$ is a Tikhonov regularization term, which may assure the strong convergence of a generated trajectory to the minimizer of minimal norm of the objective function $g.$
For further insight into the Tikhonov regularization techniques we refer to  \cite{ACR,AC,att-com1996,AL-nemkoz,BCL,BGMS,CPS}.

The case $q=1$ was intensively studied in the literature. Indeed, in \cite{ACR} Attouch, Chbani and Riahi showed that for $q=1,$ $p>2,$ $a=1$ and $\a>3$ the generated trajectories of \eqref{DynSys} converge weakly to a minimizer of $g$. Further, one has $g(x(t)-\min g=o\left(\frac{1}{t^2}\right).$ On the other hand, if $p<2$, then the strong convergence result $\liminf_{t\to+\infty}\|x(t)-x^*\|=0$ holds, where $x^*$ is the element of minimum norm  from $\argmin g.$ The previous result is also true if $p=2$ and  $a>\frac{4}{9}\a(\a-3).$ Similar results have been obtained in \cite{BCL} and \cite{AL-siopt} for some second order dynamical systems with (implicit) Hessian driven damping and Tikhonov regularization term (see also \cite{att-p-r-jde2016,ALP}). The case $p=2$ seems to be critical in the sense that separates the  case when one obtains fast convergence of the function values and weak convergence of the trajectories to a minimizer, and the case when the strong convergence of the trajectories to a minimizer of minimum norm is ensured. However, recently Attouch and L\'aszl\'o in \cite{AL-nemkoz} succeeded to obtain both rapid convergence towards the infimal value of $g$, and the strong convergence of the trajectories towards the element of minimum norm of the set of minimizers of $g.$ More precisely, in \cite{AL-nemkoz} it is shown that if $q=1,$ $p=2,$ $a>0$ and $\a>3$ then
$\displaystyle g\left(x(t)\right) - \min g =O\left(\frac{1}{t^2}\right) \mbox{ as }t\to+\infty.$ Further, the trajectory $x$ is bounded,
$\displaystyle \|\dot{x}(t)\| =O\left(\frac{1}{t}\right) \mbox{ as }t\to+\infty$, and
there is  strong convergence to the minimum norm solution $x^*$, i.e.
$ \liminf_{t \to \infty}{\| x(t) - x^\ast \|} = 0.$ A similar result has been obtained in \cite{AL-siopt} for a second order dynamical systems with implicit Hessian driven damping.

We emphasize that for $q=1$ the dynamical system \eqref{DynSys} is a Tikhonov regularized version of the second order dynamical system with vanishing damping,  studied by Su-Boyd-Cand\`es \cite{su-boyd-candes} in connection to the optimization problem $(P)$, that is,
$$\rm{(HBS)}\,\,\,\,\,\ddot{x}(t) + \frac{\a}{t} \dot{x}(t) + \nabla g\left(x(t)\right)=0,\,x(t_0) = u_0, \,\dot{x}(t_0) = v_0,\,u_0,v_0\in\mathcal{H}.$$
 It is obvious that the latter system can be obtained from \eqref{DynSys} by taking $q=1$ and $a=0$. 
 
 According to \cite{su-boyd-candes}, the trajectories generated by (HBS) assure fast minimization property of order $\mathcal{O}\left(1/t^2\right)$ for the decay $g(x(t))-\min g,$ provided $\a\ge3$. For $\alpha >3$, it has been shown by Attouch-Chbani-Peypouquet-Redont \cite{att-c-p-r-math-pr2018} that each trajectory generated by (HBS) converges weakly to a minimizer of the objective function $g$. Further, it is shown in \cite{AP} and \cite{May} that the asymptotic convergence rate of the values is actually  $o(1/t^2)$.

The case  $ \alpha = 3 $ corresponds to Nesterov's historical algorithm \cite{nesterov83} as it  was emphasized in \cite{su-boyd-candes}, more precisely for  $\alpha =3$, (HBS)  can be seen as a continuous version of the accelerated gradient method of Nesterov, (see also \cite{L}).

However, the case $ \alpha = 3 $ is critical, that is, the convergence of the trajectories generated by (HBS) remains an open problem.
 The subcritical case $\alpha \leq 3$ has  been examined by Apidopoulos-Aujol-Dossal \cite{AAD}  and Attouch-Chbani-Riahi \cite{ACR-subcrit}, with the convergence rate of the objective values $\displaystyle{\mathcal{O}\Big( t^{-\frac{2\alpha}{3}}}\Big)$.

 When the objective function $g$ is  not convex, the convergence of the  trajectories generated by  (HBS)  is a largely  open question.  Recent progress has been made in \cite{BCL1}, where the convergence of the trajectories  of a system, which can be considered as a perturbation of (HBS) has been obtained in a non-convex setting.
For other results concerning the dynamical system (HBS) and its extensions  we refer to \cite{ACR-subcrit,cabotenglergadat,JM}.

 Let us mention that for the particular case $q=\frac{p}{2},$  \eqref{DynSys} becomes (TRIGS) the dynamical system  introduced recently in \cite{AL-nemkoz} and studied further in \cite{ABCR}.

Note that in \cite{AL-nemkoz} it is shown that for $\frac23<p<2,$ $q=\frac{p}{2}$, $\a>0$ and $a=1$   a trajectory of the dynamical system \eqref{DynSys} satisfies the following:
$
g(x(t))-\min g= \mathcal O \left(  \frac{1}{t^{\frac{3p}{2}-1}}   \right)$  and $\liminf_{t \to \infty}{\| x(t) - x^\ast \|} = 0.
$
 
 The previous result has been improved in \cite{ABCR}, where the authors showed that for $0<q<1,\,p=2q,$ $\a>0$ and $a=1$ one has:
 $\displaystyle g\left(x(t)\right) - \min g =O\left(\frac{1}{t^{2q}}\right) \mbox{ as }t\to+\infty,$
$\displaystyle \|\dot{x}(t)\| =O\left(\frac{1}{t^{\frac{q+1}{2}}}\right) \mbox{ as }t\to+\infty$
 and
 $ \lim_{t \to \infty}{\| x(t) - x^\ast \|} = 0.$
 More precisely, by denoting $x_t$ the unique minimizer of the strongly convex function $g(x)+\frac{1}{2t^p}\|x\|^2$, it is shown in \cite{ABCR} that $\| x(t) - x_t \|=O\left(\frac{1}{t^{\frac{1-q}{2}}}\right) \mbox{ as }t\to+\infty$ and this implies that
$$ \lim_{t \to \infty}{\| x(t) - x^\ast \|} = 0.$$

The main  goal of this paper  is to extend and improve the results obtained in \cite{AL-nemkoz} and \cite{ABCR} for the general case $0<q<1$ and $p,\a,a>0.$
Since our analysis shows that the parameters $q,\,p$ and $a$ are strongly related in order to give a better perspective of
the results obtained in this paper, we emphasize the following.
\begin{enumerate}
\item If $0<q<1,\,0<p<q+1$   and  $x^*=\proj_{\argmin g} 0$ is the minimum norm element from $\argmin g$, then the following results hold.
\begin{itemize}
\item[(i)] $ \lim\limits_{t \to \infty} \| x(t) - x^\ast \| = 0.$

\item[(ii)] If $\frac{3q+1}{2}\le p<q+1$ then $\|\dot{x}(t)\|=O\left(\frac{1}{t^{2q-p+1}}\right)\mbox{ as }t\to+\infty.$ 
Further,  if $\frac{3q+1}{2}\le p\le \frac{4q+2}{3}$, then
$g(x(t))-\min g=O\left(\frac{1}{t^{p}}\right)\mbox{ as }t\to+\infty$ and for $\frac{4q+2}{3}< p< q+1$ one has
$g(x(t))-\min g=O\left(\frac{1}{t^{4q-2p+2}}\right)\mbox{ as }t\to+\infty.$

\item[(iii)] If $0<p< \frac{3q+1}{2}$, then
$\|\dot{x}(t)\|=O\left(\frac{1}{t^{\frac{p+1-\max(q,p-q)}{2}}}\right)\mbox{ as }t\to+\infty$
and $g(x(t))-\min g=O\left(\frac{1}{t^p}\right)\mbox{ as }t\to+\infty.$
\end{itemize}

\item  If $0<q<1,\,q+1< p<2,$ or $p=2$ and $a\ge q(1-q)$ then the trajectory $x(t)$ is bounded and following results hold.
\begin{itemize}
\item[(i)] $\displaystyle{\int_{t_0}^{+\infty} t^q\|\dot{x}(t)\|^2dt<+\infty\mbox{ and } \int_{t_0}^{+\infty} t^{q} \left( g\left(x(t)\right) - \min g\right)dt<+\infty}.$

\item[(ii)] $\displaystyle{\| \dot{x}(t) \|=o\left(\frac{1}{t^q}\right) \mbox{ as }t\to+\infty\mbox{ and } g\left(x(t)\right) - \min g =o\left(\frac{1}{t^{2q}}\right) \mbox{ as }t\to+\infty.}$

\item[(iii)] For every  $x^*\in\argmin g$,  there  exists  the limit $\lim_{t\to+\infty}\|x(t)-x^*\|.$ Even more, the trajectory  $x(t)$ converges weakly, as $t\to+\infty$, to an element of $\argmin g$.
\end{itemize}

\item The case $p=q+1$, $0<q<1$ is critical in the sense that separates the cases of weak and strong convergence of the trajectories. In this case we could not obtain any convergence result for the generated trajectories, however we show that some pointwise estimates hold.
\end{enumerate}

Observe that  our results considerably extend the results  obtained in \cite{ABCR}, where only the case $q=\frac{p}{2}$ was considered. Of course for this instance we reobtain the results from \cite{ABCR}. Further, our analysis reveals that the choice $p=2q$ in the dynamical system \eqref{DynSys} is not necessarily optimal. Indeed, for a fixed $q$ we may obtain the convergence rate of order
$O\left(\frac{1}{t^{\frac{4q+2}{3}}}\right)$ for the decay $g(x(t))-\min g$, meanwhile in \cite{ABCR} only the rate $O\left(\frac{1}{t^{2q}}\right)$ has been obtained. On the other hand, if we fix $p\in]0,2[$, then for every $q\in\left]\max\left(0,\frac{2p-1}{3}\right),1\right[$  one has $g(x(t))-\min g=O\left(\frac{1}{t^{p}}\right),$ hence we can choose infinitely many damping coefficients in order to obtain the same convergence rate as in case $q=\frac{p}{2}.$ These features of our dynamical system \eqref{DynSys} will also be underlined  via some numerical experiments.

The paper is organized as follows. In the next section we show the weak convergence of the trajectories generated by the dynamical system \eqref{DynSys} to a minimizer of the objective function $g.$ Pointwise and integral estimates for the velocity and decay $g\left(x(t)\right) - \min g$ are also obtained. In section 3 we give sufficient conditions that assure the strong convergence of the trajectories generated by the dynamical system \eqref{DynSys} to the minimum norm element from $\argmin g.$ Pointwise estimates are also obtained under the same assumptions. In section 4 we present some numerical experiment in order to give a better insight on the behaviour of a trajectory generated by the dynamical system \eqref{DynSys}. Finally we conclude our paper by emphasizing some perspectives.

\section{Asymptotic analysis of the trajectories generated by the dynamical system \eqref{DynSys}}

Existence and uniqueness of a $C^2([t_0,+\infty ),\mathcal{H })$ global solution of the dynamical system
\eqref{DynSys} can be shown via the classical Cauchy-Lipschitz-Picard theorem by rewriting
\eqref{DynSys} as a first order system in the product space $\mathcal{H} \times  \mathcal{H}$, see Theorem \ref{Cauchy-weel-posed} from Appendix. In this
section we carry out the asymptotic analysis concerning the trajectories $x(t)$ generated by
the dynamical system \eqref{DynSys}. A new feature of our analysis is that we provide the integral estimates $\int_{t_0}^{+\infty} t^q\|\dot{x}(t)\|^2dt<+\infty$ and  $\int_{t_0}^{+\infty} t^{q} \left( g\left(x(t)\right) - \min g\right)dt<+\infty.$
 The convergence rates
$\| \dot{x}(t) \|=O\left(\frac{1}{t^q}\right) \mbox{ as }t\to+\infty$ and $ g\left(x(t)\right) - \min g =O\left(\frac{1}{t^{2q}}\right) \mbox{ as }t\to+\infty$ are also obtained and it is shown that the trajectory is bounded. Based on these results we are able to show  for every $x^*\in\argmin g$  the existence of the limit $\lim_{t\to+\infty}\|x(t)-x^*\|.$ Finally, we obtain '$o$' estimates for the decay $g\left(x(t)\right) - \min g$ and the velocity $\| \dot{x}(t) \|$ and we show that a trajectory converges weakly to a minimizer of the objective function $g.$

In the next result we show that pointwise estimates can be obtained in the following general cases.

\begin{theorem}\label{Oestimates}
Assume that $0<q<1,$  $\a>0$, $0<p\le 2$ and for $p=2$ one has $a\ge q(1-q)$.
Let $t_0 > 0$ and  for some starting points $u_0,v_0\in\mathcal{H}$ let $x : [t_0, \infty) \to \mathcal{H}$ be the unique global solution of  (\ref{DynSys}).
Then, the  following results hold.
\begin{itemize}
\item[(i)] If $0<q<\frac{p}{2}\le 1$ then $g\left(x(t)\right) - \min g =O\left(\frac{1}{t^{2q}}\right) \mbox{ as }t\to+\infty\mbox{ and } \| \dot{x}(t) \|=O\left(\frac{1}{t^q}\right) \mbox{ as }t\to+\infty.$
\item[(ii)] If $\frac{p}{2}\le q< 1$ then $g\left(x(t)\right) - \min g =O\left(\frac{1}{t^{p}}\right) \mbox{ as }t\to+\infty\mbox{ and } \| \dot{x}(t) \|=O\left(\frac{1}{t^{\frac{p}{2}}}\right) \mbox{ as }t\to+\infty.$
\end{itemize}
\end{theorem}
\begin{proof}  First, let $x^\ast \in \argmin g$ and consider $b>0$ that will be adjusted later.
 We denote $g^\ast := g(x^\ast)=\min g$ and we introduce the  energy functional $\mathcal{E} : [t_0, \infty) \to \mathbb{R}$
\begin{align}\label{Lyapunov}
\mathcal{E}(t) =& t^{2q} \left( g(x(t)) - g^\ast \right) +\frac{a}{2t^{p-2q}}\|x(t)\|^2+ \frac{1}{2} \| b(x(t)-x^\ast) + t^q \dot{x}(t) \|^2+\frac{b(\a-b-qt^{q-1} )}{2}\|x(t)-x^*\|^2.
\end{align}

Then,
\begin{align}\label{engyderiv1}
\dot{\mathcal{E}}(t) &=2qt^{2q-1} \left( g(x(t)) - g^\ast \right) + t^{2q} \<\n g(x(t)),\dot{x}(t)\>\\
\nonumber&+\frac{a(2q-p)t^{2q-p-1}}{2}\|x(t)\|^2+at^{2q-p}\<\dot{x}(t),x(t)\>+\<(b+qt^{q-1})\dot{x}(t) +t^q\ddot{x}(t),b(x(t)-x^\ast)+t^q \dot{x}(t)\>\\
\nonumber&+b(\a-b-qt^{q-1} )\<\dot{x}(t),x(t)-x^*\>+\frac{bq(1-q)t^{q-2}}{2}\|x(t)-x^*\|^2.
\end{align}
From the dynamical system (\ref{DynSys}), we have that
\begin{align}\label{SecondOrderDeriv}
\ddot{x}(t) = - \frac{a}{t^p} x(t)- \frac{\alpha}{t^q} \dot{x}(t) - \nabla g(x(t)),
\end{align}
hence
\begin{align}\label{forengy1}
&\<(b+qt^{q-1})\dot{x}(t) +t^q\ddot{x}(t),b(x(t)-x^\ast)+t^q \dot{x}(t)\>=\\
\nonumber&\<(b+qt^{q-1}-\a )\dot{x}(t) -at^{q-p}x(t)-t^q\nabla g(x(t)), b(x(t)-x^\ast)+t^q \dot{x}(t)\>=\\
\nonumber& b(b+qt^{q-1}-\a )\<\dot{x}(t),x(t)-x^*\>+(b+qt^{q-1}-\a )t^q\|\dot{x}(t)\|^2\\
\nonumber&-abt^{q-p}\<x(t),x(t)-x^*\>-at^{2q-p}\<\dot{x}(t),x(t)\>\\
\nonumber& -b t^q\<\n g(x(t)), x(t)-x^\ast\>-t^{2q}\<\nabla g(x(t)), \dot{x}(t)\>.
\end{align}

Combining \eqref{engyderiv1} and \eqref{forengy1} we get

\begin{align}\label{engyderiv2}
\dot{\mathcal{E}}(t) &= 2qt^{2q-1}\left( g(x(t)) - g^\ast \right)+(b+qt^{q-1}-\a )t^q\|\dot{x}(t)\|^2+\frac{a(2q-p)t^{2q-p-1}}{2}\|x(t)\|^2\\
\nonumber&-b t^q\<\n g(x(t) )+a t^{-p}x(t), x(t)-x^\ast\>+\frac{bq(1-q)t^{q-2}}{2}\|x(t)-x^*\|^2.
\end{align}

Consider now the strongly convex function
$$g_t:\mathcal{H}\To\R,\,g_t(x)=g(x)+\frac{a}{2t^{p}}\|x\|^2.$$
From the gradient inequality we have
$$g_t(y)-g_t(x)\ge\<\n g_t(x),y-x\>+\frac{a}{2t^{p}}\|x-y\|^2,\mbox{ for all }x,y\in\mathcal{H}.$$
Take now $y=x^*$ and $x=x(t).$ We get
\begin{align*}
&g(x^*)+\frac{a}{2t^{p}}\|x^*\|^2-g(x(t))-\frac{a}{2t^{p}}\|x(t)\|^2\ge\\
&-\<\n g(x(t))+at^{-p}x(t),x(t)-x^*\>+\frac{a}{2t^{p}}\|x(t)-x^*\|^2.
\end{align*}
Consequently,
\begin{align}\label{insert1}
-bt^q\<\n g(x(t))+at^{-p}x(t),x(t)-x^*\>&\le bt^q(g(x^*)-g(x(t)))+\frac{abt^{q-p}}{2}\|x^*\|^2-\frac{abt^{q-p}}{2}\|x(t)\|^2\\
\nonumber&-\frac{abt^{q-p}}{2}\|x(t)-x^*\|^2.
\end{align}

By inserting \eqref{insert1} in \eqref{engyderiv2} we get
\begin{align}\label{engyderiv3}
\dot{\mathcal{E}}(t) &= (2qt^{2q-1}-bt^q) \left( g(x(t)) - g^\ast \right)+(b+qt^{q-1}-\a )t^q\|\dot{x}(t)\|^2+\frac{a(2q-p)t^{2q-p-1}-abt^{q-p}}{2}\|x(t)\|^2\\
\nonumber&+\frac{bq(1-q)t^{q-2}-abt^{q-p}}{2}\|x(t)-x^*\|^2+\frac{abt^{q-p}}{2}\|x^*\|^2.
\end{align}

Since
$$\frac{1}{2} \| b(x(t)-x^\ast) + t^q \dot{x}(t) \|^2\le b^2\|x(t)-x^*\|^2+t^{2q}\|\dot{x}(t)\|^2$$
one has
\begin{align}\label{forgronwal1}
\mathcal{E}(t) \le & t^{2q} \left( g(x(t)) - g^\ast \right) +\frac{a}{2t^{p-2q}}\|x(t)\|^2+ t^{2q} \|  \dot{x}(t) \|^2+\frac{b(\a+b-qt^{q-1})}{2}\|x(t)-x^*\|^2.
\end{align}
Let $r=\max(q,p-q)$ and consider $K>0$ that will be defined in what follows.
Now, by multiplying \eqref{forgronwal1} with $\frac{K}{t^r}$ and adding to \eqref{engyderiv3} we get
\begin{align}\label{engyderivforgronwall}
\dot{\mathcal{E}}(t)+ \frac{K}{t^r}\mathcal{E}(t)&\le (2qt^{2q-1}-bt^q+Kt^{2q-r}) \left( g(x(t)) - g^\ast \right)+(b+qt^{q-1}-\a +Kt^{q-r} )t^q\|\dot{x}(t)\|^2\\
\nonumber&+\frac{a(2q-p)t^{2q-p-1}-abt^{q-p}+aKt^{2q-p-r}}{2}\|x(t)\|^2\\
\nonumber&+\frac{bq(1-q)t^{q-2}-abt^{q-p}+Kb(\a+b-qt^{q-1})t^{-r}}{2}\|x(t)-x^*\|^2\\
\nonumber&+\frac{abt^{q-p}}{2}\|x^*\|^2.
\end{align}

Now, take  $0<b<\a$ . If $p<2$ consider $0<K<\min\left(b,\a-b,\frac{a}{\a+b}\right).$ If $p=2$ then by hypotheses we have $a>q(1-q)$, so take $0<K<\min\left(b,\a-b,\frac{a-q(1-q)}{\a+b}\right).$
Then, easily can be observed that there exists $t_1\ge t_0$ such that
\begin{align}\label{engyderivforgronwal2}
\dot{\mathcal{E}}(t)+ \frac{K}{t^r}\mathcal{E}(t)&\le\frac{abt^{q-p}}{2}\|x^*\|^2\mbox{ for all }t\ge t_1.
\end{align}
Now, since $r=\max(q,p-q)$ and $0<q< 1$ we can take $r=1$ provided $p= q+1.$ In this case, by multiplying \eqref{engyderivforgronwal2} with
$t^K$ we get
$$\frac{d}{dt}\left(t^K\mathcal{E}(t)\right)\le \frac{ab\|x^*\|^2}{2}t^{K-1}\mbox{ for all }t\ge t_1.$$
By integrating the latter relation on an interval $[t_1,T],\,T>t_1$ we conclude that there exists $C>0$ such that
$$T^K\mathcal{E}(T)\le \frac{ab\|x^*\|^2}{2K}T^{K}+C.$$
Consequently,
$$\mathcal{E}(T)\le \frac{ab\|x^*\|^2}{2K}+\frac{C}{T^K}\le C_1\mbox{ for some }C_1>0,$$
hence for all $t\ge t_1$ one has
$$g(x(t))-g^*\le \frac{C_1}{t^{2q}}\mbox{ and }\|\dot{x}(t)\|\le\frac{\sqrt{C_1}}{t^{q}}.$$

In other words
\begin{equation}\label{qlessrate}
g(x(t))-g^*=O\left(\frac{1}{t^{2q}}\right)\mbox{ and }\|\dot{x}(t)\|=O\left(\frac{1}{t^q}\right),\mbox{ as }t\to+\infty.
\end{equation}

In the case $r\neq 1$, by multiplying \eqref{engyderivforgronwal2} with
$e^{\frac{K}{1-r}t^{1-r}}$ we get
\begin{equation}\label{foru}\frac{d}{dt}\left(e^{\frac{K}{1-r}t^{1-r}}\mathcal{E}(t)\right)\le \frac{ab\|x^*\|^2}{2}t^{q-p}e^{\frac{K}{1-r}t^{1-r}}\mbox{ for all }t\ge t_1.
\end{equation}
Observe further, that for all $t\ge t_1$ one has
$$\frac1K \frac{d}{dt}\left( t^{q-p+r}e^{\frac{K}{1-r}t^{1-r}}\right)=\frac1K\left((q-p+r) t^{q-p+r-1}+Kt^{q-p}\right)e^{\frac{K}{1-r}t^{1-r}}\ge t^{q-p}e^{\frac{K}{1-r}t^{1-r}}.$$
Combining the latter relation with \eqref{foru} we get
\begin{equation}\label{foru1}
\frac{d}{dt}\left(e^{\frac{K}{1-r}t^{1-r}}\mathcal{E}(t)\right)\le \frac{ab\|x^*\|^2}{2K}\frac{d}{dt}\left( t^{q-p+r}e^{\frac{K}{1-r}t^{1-r}}\right)\mbox{ for all }t\ge t_1.
\end{equation}

By integrating \eqref{foru1} on an interval $[t_1,T],\,T>t_1$ we conclude that
$$e^{\frac{K}{1-r}T^{1-r}}\mathcal{E}(T)\le C_1 T^{q-p+r}e^{\frac{K}{1-r}T^{1-r}}+C_2,$$
where $C_1=\frac{ab\|x^*\|^2}{2K}$ and $C_2=e^{\frac{K}{1-r}t_1^{1-r}}\mathcal{E}(t_1)-\frac{ab\|x^*\|^2}{2K}t_1^{q-p+r}e^{\frac{K}{1-r}t_1^{1-r}}.$

Consequently,
$$\mathcal{E}(T)\le C_1 T^{q-p+r}+\frac{C_2}{e^{\frac{K}{1-r}T^{1-r}}}\le C_3T^{q-p+r}\mbox{ for some }C_3>0,$$
hence for all $t\ge t_1$ one has
\begin{equation}\label{rno1rate}
g(x(t))-g^*\le \frac{C_3}{t^{q+p-r}}\mbox{ and }\|\dot{x}(t)\|\le\frac{\sqrt{C_3}}{t^{\frac{q+p-r}{2}}}.
\end{equation}

Now,  since $r=\max(q,p-q)$ one has  $r=q$ if $2q\ge p$ and $r=p-q$ if $p\ge 2q.$

Therefore, in case  $q<1$ and $2q\ge p$, \eqref{rno1rate} leads to
$$g(x(t))-g^*=O\left( \frac{1}{t^{p}}\right)\mbox{ and }\|\dot{x}(t)\|=O\left(\frac{1}{t^{\frac{p}{2}}}\right)\mbox{ as }t\to+\infty.$$
Finally, in case $q<1$ and $2q< p,\,p\neq q+1$, \eqref{rno1rate} leads to
$$g(x(t))-g^*=O\left( \frac{1}{t^{2q}}\right)\mbox{ and }\|\dot{x}(t)\|=O\left(\frac{1}{t^q}\right)\mbox{ as }t\to+\infty.$$

\end{proof}
\begin{remark} Note that, though provides the same convergence rate, the case $p=q+1$ must be treated separately in the proof of Theorem \ref{Oestimates}. This is due to the fact that in this case $r=1$, hence the antiderivative of $\frac{K}{t^r}$ is $K\ln t$ and therefore \eqref{engyderivforgronwal2} needs a special attention. Actually, we will show that the case $p=q+1$ is critical, in the sense that separates the cases when the trajectories converge strongly and the trajectories converges weakly, respectively.
\end{remark}

\begin{remark}\label{r3} Observe  that our analysis also work for the case $q=1$ and  $\a>3.$ Indeed, in this case by taking $2<b<\a-1$ and  $0<K<\min\left(b-2,\a-1-b,\frac{a}{\a+b-1}\right)$ in the proof of Theorem \ref{Oestimates} we reobtain some results from \cite{ACR} and \cite{AL-nemkoz}.
More precisely if $p\le 2$ one has
$$g\left(x(t)\right) - \min g =O\left(\frac{1}{t^{p}}\right) \mbox{ as }t\to+\infty\mbox{ and } \| \dot{x}(t) \|=O\left(\frac{1}{t^{\frac{p}{2}}}\right) \mbox{ as }t\to+\infty.$$

Further, if $p\ge 2$ one has
$$g\left(x(t)\right) - \min g =O\left(\frac{1}{t^{2}}\right) \mbox{ as }t\to+\infty\mbox{ and } \| \dot{x}(t) \|=O\left(\frac{1}{t}\right) \mbox{ as }t\to+\infty.$$
\end{remark}

In the next result we obtain some integral estimates in case $p>q+1.$ 

\begin{theorem}\label{integralOestimates}
Assume that $0<q\le 1,$  $q+1<p\le 2$ and for $p=2$ one has $a\ge q(1-q)$.
Let $t_0 > 0$ and  for some starting points $u_0,v_0\in\mathcal{H}$ let $x : [t_0, \infty) \to \mathcal{H}$ be the unique global solution of  (\ref{DynSys}).
Then, the trajectory $x(t)$ is bounded and following results hold.

 (integral estimates) \quad
$\displaystyle{\int_{t_0}^{+\infty} t^q\|\dot{x}(t)\|^2dt<+\infty\mbox{ and } \int_{t_0}^{+\infty} t^{q} \left( g\left(x(t)\right) - \min g\right)dt<+\infty}.$
\medskip

 (pointwise estimates)
$\displaystyle{\| \dot{x}(t) \|=O\left(\frac{1}{t^q}\right) \mbox{ as }t\to+\infty\mbox{ and } g\left(x(t)\right) - \min g =O\left(\frac{1}{t^{2q}}\right) \mbox{ as }t\to+\infty.}$
\end{theorem}

\begin{proof}
Consider the energy functional defined at \eqref{Lyapunov}. 
According to \eqref{engyderiv3} one has
\begin{align}\label{engyderiv31}
\dot{\mathcal{E}}(t) &= (2qt^{2q-1}-bt^q) \left( g(x(t)) - g^\ast \right)+(b+qt^{q-1}-\a )t^q\|\dot{x}(t)\|^2+\frac{a(2q-p)t^{2q-p-1}-abt^{q-p}}{2}\|x(t)\|^2\\
\nonumber&+\frac{bq(1-q)t^{q-2}-abt^{q-p}}{2}\|x(t)-x^*\|^2+\frac{abt^{q-p}}{2}\|x^*\|^2.
\end{align}

Now,  we take $b<\a$  and we conclude that there exist $t_1\ge t_0$ and $C_1,C_2>0$ such that for all $t\ge t_1$ the following hold.
 $$\mathcal{E}(t)\ge 0,$$
$$2qt^{2q-1}-bt^q\le -C_1 t^q,$$
$$(b+qt^{q-1}-\a )t^q\le -C_2 t^q$$
and
$$\frac{a(2q-p)t^{2q-p-1}-abt^{q-p}}{2}\|x(t)\|^2+\frac{bq(1-q)t^{q-2}-abt^{q-p}}{2}\|x(t)-x^*\|^2\le 0.$$
Consequently,
\begin{equation}\label{e1}
\dot{\mathcal{E}}(t) +C_1t^q \left( g(x(t)) - g^\ast \right)+C_2t^q\|\dot{x}(t)\|^2\le \frac{ab\|x^*\|^2}{2}\frac{1}{t^{p-q}},\mbox{ for all }t\ge t_1.
\end{equation}

By integrating \eqref{e1} on $[t_1,T],\, T>t_1$ and taking into account that by hypotheses we have $p>q+1$, we conclude that there exists $C_3>0$ such that
\begin{align}\label{e2}
\int_{t_1}^T\dot{\mathcal{E}}(t)dt +C_1\int_{t_1}^T t^q \left( g(x(t)) - g^\ast \right)dt+C_2\int_{t_1}^T t^q\|\dot{x}(t)\|^2dt&\le \frac{ab\|x^*\|^2}{2}\int_{t_1}^T \frac{1}{t^{p-q}}dt\\
\nonumber&\le\frac{ab\|x^*\|^2}{2}\int_{t_1}^{+\infty} \frac{1}{t^{p-q}}dt\le C_3.
\end{align}
Hence, $\mathcal{E}(t)$ is bounded which implies that $x(t)$ is bounded,
$$g(x(t)-g^*=O\left(\frac{1}{t^{2q}}\right)\mbox{ as }t\to+\infty$$
and
$$\|\dot{x}(t)\|=O\left(\frac{1}{t^{q}}\right)\mbox{ as }t\to+\infty.$$
Further, \eqref{e2} ensures that
\begin{equation}\label{L1est}
t^q \left( g(x(t)) - g^\ast \right),t^q\|\dot{x}(t)\|^2\in L^1[t_0,+\infty[.
\end{equation}
\end{proof}

\begin{remark} Observe that the conclusion of the previous theorem remains valid also in case $q=1,\,p>2$ and $\a>3$ and this fact can be seen if one takes $2<b<\a-1$ in its proof. However, these results have already been obtained in \cite{ACR}.
 
 Even more, the pointwise estimates remain valid also  in the case $q=1$ $p>2$ and $\a=3$. Indeed, the result follows if one  takes $b=2$ in the proof of Theorem \ref{integralOestimates}, however the integral estimates do not hold anymore.
\end{remark}

Now, we are able to show the existence of the limit $\lim_{t\to+\infty}\|x(t)-x^*\|.$ 

\begin{lemma}\label{xlimit} Assume that $q+1<p\le 2$  and for  $p=2$ one has $a\ge q(1-q).$ 
For some starting points $u_0,v_0\in\mathcal{H}$ let $x : [t_0, \infty) \to \mathcal{H}$ be the unique global solution of  (\ref{DynSys}). Let $x^*\in\argmin g.$ Then,  there  exists  the limit $\lim_{t\to+\infty}\|x(t)-x^*\|.$
\end{lemma}
\begin{proof} The proof is based on Lemma \ref{A1}. Indeed, for $x^*\in \argmin g$ consider the function $w(t)=\frac12\|x(t)-x^*\|^2.$ Then, by using \eqref{DynSys} one has
$$\ddot{w}(t)+\frac{\a}{t^q}\dot{w}(t)=\left\<\ddot{x}(t)+\frac{\a}{t^q}\dot{x}(t),x(t)-x^*\right\>+\|\dot{x}(t)\|^2=\left\<-\n g({x}(t))-\frac{a}{t^p}{x}(t),x(t)-x^*\right\>+\|\dot{x}(t)\|^2$$
Now, from the monotonicity of $\n g$ we have $\<-\n g({x}(t)),x(t)-x^*\>\le 0$ and
$$\left\<-\frac{a}{t^p}{x}(t),x(t)-x^*\right\>=\frac{a}{2t^p}(\|x^*\|^2-\|x(t)\|^2-\|x(t)-x^*\|^2).$$
Consequently,
\begin{equation}\label{foruse}
\ddot{w}(t)+\frac{\a}{t^q}\dot{w}(t)\le \frac{a}{2t^p}\|x^*\|^2+\|\dot{x}(t)\|^2.
\end{equation}
We show that Lemma \ref{A1}  can be applied with $p(t)=1,\,q(t)=\frac{\a}{t^q}$ and $k(t)=\frac{a}{2t^p}\|x^*\|^2+\|\dot{x}(t)\|^2.$
Note that since $p>q+1$ and according to Theorem \ref{integralOestimates} it holds that $t^q\|\dot{x}(t)\|^2\in L^1[t_0,+\infty[,$ we have
\begin{equation}\label{forexist1}
t^qk(t)\in L^1[t_0,+\infty[.
\end{equation}

Now, if $0<q<1$ then $\int_{t_0}^t \frac{q(s)}{p(s)}ds=\frac{\a}{1-q}(t^{1-q}-t_0^{1-q}),$ hence
$\exp\left(\int_{t_0}^t \frac{q(s)}{p(s)}ds\right)=C_1e^{C_2t^{1-q}},$ where $C_1=e^{-\frac{\a}{1-q}t_0^{1-q}}$ and $C_2=\frac{\a}{1-q}.$
Consequently,
\begin{equation}\label{firstcond}
\frac{1}{\exp\left(\int_{t_0}^t \frac{q(s)}{p(s)}ds\right)}\in L^1[t_0,+\infty[.
\end{equation}
Further
\begin{align}\label{secondcond1}
\left(\int_{t}^{+\infty}\frac{dT}{\exp\left(\int_{t_0}^T \frac{q(s)}{p(s)}ds\right)}\right)\frac{\exp\left(\int_{t_0}^t\frac{q(s)}{p(s)}ds\right)}{p(t)}k(t)=\int_{t}^{+\infty}e^{-C_2T^{1-q}}dT e^{C_2t^{1-q}}k(t).
\end{align}
Consider the integral
$I_t=\int_{t}^{+\infty}e^{-C_2T^{1-q}}dT.$ By substituting $C_2T^{1-q}=\t$ we obtain $T=C_2^{-\frac{1}{1-q}}\t^{\frac{1}{1-q}},$ hence
$dT=C_3\t^{\frac{q}{1-q}}d\t,$ where $C_3=\frac{C_2^{-\frac{1}{1-q}}}{1-q}.$ Consequently,
$$I_t=C_3\int_{C_2t^{1-q}}^{+\infty}\t^{\frac{q}{1-q}}e^{-\t}d\t=C_3\Gamma\left(\frac{1}{1-q},C_2t^{1-q}\right),$$
where $\Gamma(s,x)=\int_{x}^{+\infty}\t^{s-1}e^{-\t}d\t$ is the upper incomplete gamma function. It is well known (see \cite{Milst}) that
$$\frac{\Gamma(s,x)}{x^{s-1}e^{-x}}\to1\mbox{ as }x\to+\infty.$$
Consequently,
$$\frac{\Gamma\left(\frac{1}{1-q},C_2t^{1-q}\right)}{(C_2t^{1-q})^{\frac{q}{1-q}}e^{-C_2t^{1-q}}}\to 1,\mbox{ as }t\to+\infty,$$
which shows that there exists $C_4>0$ such that
$$\int_{t}^{+\infty}e^{-C_2T^{1-q}}dT e^{C_2t^{1-q}}\le C_4 t^q .$$
Now, combining the previous relation with  \eqref{forexist1} and \eqref{secondcond1} we obtain that
\begin{equation}\label{secondcond}
\left(\int_{t}^{+\infty}\frac{dT}{\exp\left(\int_{t_0}^T \frac{q(s)}{p(s)}ds\right)}\right)\frac{\exp\left(\int_{t_0}^t\frac{q(s)}{p(s)}ds\right)}{p(t)}k(t)\le C_4 t^q k(t)\in L^1[t_0,+\infty[.
\end{equation}
Hence, according to Lemma \ref{A1}, there exists the limit $\lim_{t\to+\infty}\|x(t)-x^*\|.$
\end{proof}

In the next results we obtain '$o$' estimates for the decay $g(x(t))-\min g$ and the velocity $\dot{x}(t),$ further we show that a trajectory generated by the dynamical system \eqref{DynSys} converges weakly to a minimizer of the objective function $g.$

\begin{theorem}\label{oestimates}
Assume that $q+1<p\le 2$  and for  $p=2$ one has $a\ge q(1-q).$ Let $t_0 > 0$ and  for some starting points $u_0,v_0\in\mathcal{H}$ let $x : [t_0, \infty) \to \mathcal{H}$ be the unique global solution of  (\ref{DynSys}).
Then, the trajectory  $x(t)$ converges weakly, as $t\to+\infty$, to an element of $\argmin g$. Further, one has
$$\displaystyle{\| \dot{x}(t) \|=o\left(\frac{1}{t^q}\right) \mbox{ as }t\to+\infty\mbox{ and } g\left(x(t)\right) - \min g =o\left(\frac{1}{t^{2q}}\right) \mbox{ as }t\to+\infty.}$$
\end{theorem}

\begin{proof} 

 By using the same notations as in the proof of Theorem \ref{integralOestimates}, in what follows we show that there exists
$$\lim_{t\to+\infty}\left(t^{2q} \left( g(x(t)) - g^\ast \right) + \frac{t^{2q}}{2} \|\dot{x}(t) \|^2\right)\in\R.$$

Now, according to \eqref{Lyapunov} one has
\begin{align}\label{forwlimit1}
\dot{\mathcal{E}}(t) &=2qt^{2q-1} \left( g(x(t)) - g^\ast \right) + t^{2q} \<\n g(x(t)),\dot{x}(t)\>\\
\nonumber&+\frac{d}{dt}\left(\frac{a}{2t^{p-2q}}\|x(t)\|^2+ \frac{1}{2} \| b(x(t)-x^\ast) + t^q \dot{x}(t) \|^2+\frac{b(\a-b-qt^{q-1} )}{2}\|x(t)-x^*\|^2\right).
\end{align}
Further, by using \eqref{DynSys} we get
\begin{align}\label{forwlimit2}
t^{2q} \<\n g(x(t)),\dot{x}(t)\>&=-t^{2q} \<\ddot{x}(t),\dot{x}(t)\>-\a t^{q} \|\dot{x}(t)\|^2-at^{2q-p}\<x(t),\dot{x}(t)\>\\
\nonumber& =-\frac{d}{dt}\left(\frac{t^{2q}}{2}\|\dot{x}(t)\|^2\right)+qt^{2q-1}\|\dot{x}(t)\|^2-\a t^{q} \|\dot{x}(t)\|^2\\
\nonumber&-a\frac{d}{dt}\left(\frac{t^{2q-p}}{2}\|x(t)\|^2\right)+a\frac{2q-p}{2}t^{2q-p-1}\|x(t)\|^2.
\end{align}

Combining \eqref{forwlimit1}, \eqref{forwlimit2} and \eqref{e1} we get
\begin{align}\label{forwlimit3}
&\frac{d}{dt}\left(-\frac{t^{2q}}{2}\|\dot{x}(t)\|^2+ \frac{1}{2} \| b(x(t)-x^\ast) + t^q \dot{x}(t) \|^2+\frac{b(\a-b-qt^{q-1} )}{2}\|x(t)-x^*\|^2\right)\le\\
\nonumber&-2qt^{2q-1} \left( g(x(t)) - g^\ast \right)+(\a t^q-qt^{2q-1})\|\dot{x}(t)\|^2-a\frac{2q-p}{2}t^{2q-p-1}\|x(t)\|^2+\frac{ab\|x^*\|^2}{2}\frac{1}{t^{p-q}}.
\end{align}

Since $q>2q-1$, $2q-p-1<q-p<-1$ and $x(t)$ is bounded, further,  $t^q \left( g(x(t)) - g^\ast \right),t^q\|\dot{x}(t)\|^2\in L^1[t_0,+\infty[$, we get that the right hand side of \eqref{forwlimit3} is of class $L^1[t_0,+\infty[$. Hence, by Lemma \ref{fejer-cont1} there exists the limit
$$\lim_{t\to+\infty}\left(-\frac{t^{2q}}{2}\|\dot{x}(t)\|^2+ \frac{1}{2} \| b(x(t)-x^\ast) + t^q \dot{x}(t) \|^2+\frac{b(\a-b-qt^{q-1} )}{2}\|x(t)-x^*\|^2\right)\in\R.$$

Now according to Lemma \ref{xlimit} there exists $\lim_{t\to+\infty}\|x(t)-x^*\|\in\R,$ consequently there exists the limit
\begin{equation}\label{forwlimit4}
\lim_{t\to+\infty}t^q\< x(t)-x^\ast,\dot{x}(t)\>\in\R.
\end{equation}

According to \eqref{e1} one has $\dot{\mathcal{E}}(t) \le \frac{ab\|x^*\|^2}{2}\frac{1}{t^{p-q}},\mbox{ for all }t\ge t_1$ and, since $p>q+1$, the right hand side of the previous inequality is of class $L^1[t_0,+\infty[.$ Consequently Lemma \ref{fejer-cont1} assures the existence and finiteness of the limit
\begin{align*}&\lim_{t\to+\infty}\mathcal{E}(t)=\\
\nonumber&\lim_{t\to+\infty}\left(t^{2q} \left( g(x(t)) - g^\ast \right) +\frac{a}{2t^{p-2q}}\|x(t)\|^2+ \frac{1}{2} \| b(x(t)-x^\ast) + t^q \dot{x}(t) \|^2+\frac{b(\a-b-qt^{q-1} )}{2}\|x(t)-x^*\|^2\right).
\end{align*}
Obviously $p>q+1>2q$ and the fact that $x(t)$ is bounded assure that $\lim_{t\to+\infty}\frac{a}{2t^{p-2q}}\|x(t)\|^2=0.$ Further the existence of $\lim_{t\to+\infty}\|x(t)-x^*\|$ and \eqref{forwlimit4} implies that there exists the limit
\begin{equation}\label{forwlimit5}
\lim_{t\to+\infty}\left(t^{2q} \left( g(x(t)) - g^\ast \right) + \frac{t^{2q}}{2} \| \dot{x}(t) \|^2\right)\in\R.
\end{equation}
According to \eqref{L1est} we have
\begin{equation}\label{forwlimit6}
\int_{t_0}^{+\infty}\frac{1}{t^q}\left(t^{2q} \left( g(x(t)) - g^\ast \right) + \frac{t^{2q}}{2} \| \dot{x}(t) \|^2\right)dt=
\int_{t_0}^{+\infty}t^{q} \left( g(x(t)) - g^\ast \right)dt + \int_{t_0}^{+\infty}\frac{t^{q}}{2} \| \dot{x}(t) \|^2dt<+\infty.
\end{equation}
Now, since $\frac{1}{t^q}\not\in L^1[t_0,+\infty[$ from \eqref{forwlimit5} and \eqref{forwlimit6} we obtain that
$$\lim_{t\to+\infty}\left(t^{2q} \left( g(x(t)) - g^\ast \right) + \frac{t^{2q}}{2} \| \dot{x}(t) \|^2\right)=0,$$
hence,
$$g(x(t)) - g^\ast=o\left(\frac{1}{t^{2q}}\right)\mbox{ and }\| \dot{x}(t) \|=o\left(\frac{1}{t^{q}}\right).$$
\vskip0.3cm
Next we show that $x(t)$  converges weakly to a minimizer of $g.$ According to Lemma \ref{xlimit}, for every $x^*\in\argmin g$ the  limit $\lim_{t\to +\infty}\| x(t) -x^*\|$ exists. Now, if $\ol x \in  \mathcal{H}$ is  a weak sequential limit point of $x(t)$ then there exists a
sequence $(t_n)_{n\in \N} \subseteq  [t_0,+\infty[$ such that $\lim_{n\to+ \infty}  t_n = +\infty$  and $x(t_n)$ converges weakly
to $\ol x$ as $n\to+\infty.$
Obviously the function $g$ is weakly lower semicontinuous, since it is convex
and continuous, consequently $ g(\ol x) \leq  \liminf_{n\to+\infty}  g(x(t_n)=\lim_{n\to+\infty}  g(x(t_n)=g^*= \min g$,  which shows that $\ol x \in  \argmin g.$
According to the Opial lemma it follows that
$$w-\lim_{t\to+\infty}x(t) \in  \argmin g.$$
\end{proof}

\section{Strong convergence}

We continue the present section by emphasizing the main idea behind the Tikhonov regularization, which will generate strong convergence results for our dynamical system \eqref{DynSys} to a minimizer of the objective function of minimal norm. By  $x_{t}$ we denote the unique solution of the strongly convex minimization problem
\begin{align*}
 \min_{x \in \mathcal{H}} \left( g(x) + \frac{a}{2t^p} \| x \|^2 \right).
\end{align*}
We know, (see for instance \cite{att-com1996}), that the Tikhonov approximation curve $t \to x_{t}$ satisfies $x^\ast = \lim\limits_{t \to +\infty} x_{t}$, where $x^\ast = \argmin\limits_{x \in \argmin g} \| x \|$ is the minimal norm element from the set $\argmin g.$ Obviously, $\{x^*\}=\proj_{\argmin g} 0$ and we have the inequality $\| x_{t} \| \leq \| x^\ast \|$ (see \cite{BCL}).

Since $x_{t}$ is the unique minimum of the strongly convex function $g_t(x)=g(x)+\frac{a}{2t^p}\|x\|^2,$ obviously one has
\begin{equation}\label{fontos0}
\n g_t(x_{t})=\n g(x_{t})+\frac{a}{t^p}x_{t}=0.
\end{equation}
Further, according to Lemma 2 from \cite{ABCR} the function $t\To x_{t}$ is derivable almost everywhere and one has
\begin{equation}\label{fontos1}
\left\|\frac{d}{dt}x_{t}\right\|\le\frac{p}{t}\|x_{t}\|\mbox{ for almost every} t\ge t_0.
\end{equation}

Note that since $g_t$ is strongly convex, from the gradient inequality we have
\begin{equation}\label{fontos2}
g_t(y)-g_t(x)\ge\<\n g_t(x),y-x\>+\frac{a}{2t^p}\|x-y\|^2,\mbox{ for all }x,y\in\mathcal{H}.
\end{equation}
In particular
\begin{equation}\label{fontos3}
g_t(x)-g_t(x_t)\ge\frac{a}{2t^p}\|x-x_t\|^2,\mbox{ for all }x\in\mathcal{H}.
\end{equation}
Let $y:[t_0,+\infty[\to\mathcal{H}, s\To y(s)$ derivable at a point $t\in]t_0,+\infty[.$ Then it is obvious that
\begin{equation}\label{fontos4}
\frac{d}{dt}g_t(y(t))=\<\n g_t(y(t)),\dot{y}(t)\>-\frac{ap}{2t^{p+1}}\|y(t)\|^2.
\end{equation}

Finally, for all $x,y\in\mathcal{H}$, one has

\begin{equation}\label{fontos5}
g(x)-g(y)=(g_t(x)-g_t(x_t))+(g_t(x_t)-g_t(y))+\frac{a}{2t^p}(\|y\|^2-\|x\|^2)\le g_t(x)-g_t(x_t)+\frac{a}{2t^p}\|y\|^2.
\end{equation}
Now, in order to show the strong convergence of the dynamical system \eqref{DynSys} to an element of minimum norm of the nonempty, convex and closed set $\argmin g$, we state our main result of the present section.

\begin{theorem}\label{StrongConvergenceResultLyapunov}

Assume that $0<q<1,\,0<p< q+1$   and let $x$ be the unique global solution of \eqref{DynSys}. Let $x^*=\proj_{\argmin g} 0.$
Then,  $ \lim\limits_{t \to \infty} \| x(t) - x^\ast \| = 0$ and the following estimates hold.
\begin{itemize}
\item[(i)] If $\frac{3q+1}{2}\le p<q+1$ then $\|\dot{x}(t)\|=O\left(\frac{1}{t^{2q-p+1}}\right)\mbox{ as }t\to+\infty.$\\
Further,  if $\frac{3q+1}{2}\le p\le \frac{4q+2}{3}$, then
$g(x(t))-\min g=O\left(\frac{1}{t^{p}}\right)\mbox{ as }t\to+\infty$ and for $\frac{4q+2}{3}< p< q+1$ one has
$g(x(t))-\min g=O\left(\frac{1}{t^{4q-2p+2}}\right)\mbox{ as }t\to+\infty.$

\item[(ii)] If $0<p< \frac{3q+1}{2}$, then 
$\|\dot{x}(t)\|=O\left(\frac{1}{t^{\frac{p+1-\max(q,p-q)}{2}}}\right)\mbox{ as }t\to+\infty.$\\
Further, $g(x(t))-\min g=O\left(\frac{1}{t^p}\right)\mbox{ as }t\to+\infty.$
\end{itemize}
\end{theorem}

\begin{proof}

 Consider $\a>b>0$ and define, for every $t \geq t_0$, the following energy functional

\begin{align}\label{Lyapunovstr}
E(t)&=t^{2q}(g_t(x(t))-g_t(x_t))+\frac{1}{2} \| b(x(t)-x_t) + t^q \dot{x}(t) \|^2+\frac{b(\a-b-qt^{q-1} )}{2}\|x(t)-x_t\|^2.
\end{align}
Obviously, there exists $t_1\ge t_0$ such that $\frac{b(\a-b-qt^{q-1} )}{2}>0$ for all $t\ge t_1$, hence $E(t)\ge0$ for all $t\ge t_1.$
Now, by using \eqref{fontos4} and the fact that $\n g_t(x_t)=0$ we get
\begin{align}\label{strderiv}
\dot{E}(t)&=2qt^{2q-1}(g_t(x(t))-g_t(x_t))+ t^{2q}\left(\<\n g_t(x(t)),\dot{x}(t)\>-\frac{ap}{2t^{p+1}}\|x(t)\|^2+\frac{ap}{2t^{p+1}}\|x_t\|^2\right)\\
\nonumber&+\left\<b\left(\dot{x}(t)-\frac{d}{dt}x_t\right)+qt^{q-1}\dot{x}(t)+t^q\ddot{x}(t),b(x(t)-x_t)+t^q\dot{x}(t)\right\>\\
\nonumber& +\frac{b(q(1-q)t^{q-2} )}{2}\|x(t)-x_t\|^2+b(\a-b-qt^{q-1} )\left\<\dot{x}(t)-\frac{d}{dt}x_t,x(t)-x_t\right\>.
\end{align}

Proceeding as in \eqref{forengy1} we obtain

\begin{align}\label{forstrderiv1}
&\<(b+qt^{q-1})\dot{x}(t) +t^q\ddot{x}(t),b(x(t)-x_t)+t^q \dot{x}(t)\>= b(b+qt^{q-1}-\a )\<\dot{x}(t),x(t)-x_t\>\\
\nonumber&+(b+qt^{q-1}-\a )t^q\|\dot{x}(t)\|^2 -b t^q\<\n g_t(x(t)), x(t)-x_t\>-t^{2q}\<\nabla g_t(x(t)), \dot{x}(t)\>.
\end{align}

Now, according to \eqref{fontos2} one has
\begin{equation}\label{forstrderiv2}
-b t^q(g_t(x(t))-g_t(x_t))-\frac{ab}{2t^{p-q}}\|x(t)-x_t\|^2\ge b t^q\<\n g_t(x(t)),x_t-x(t)\>
\end{equation}

Combining \eqref{strderiv}, \eqref{forstrderiv1} and \eqref{forstrderiv2} we get

\begin{align}\label{strderiv1}
\dot{E}(t)&\le (2qt^{2q-1}-b t^q)(g_t(x(t))-g_t(x_t))- \frac{ap}{2t^{p-2q+1}}\left(\|x(t)\|^2-\|x_t\|^2\right)\\
\nonumber& +(b+qt^{q-1}-\a )t^q\|\dot{x}(t)\|^2+\frac{b(q(1-q)t^{q-2}-at^{q-p} )}{2}\|x(t)-x_t\|^2
\\
\nonumber&+b(\a-qt^{q-1})\left\<\frac{d}{dt}x_t,x(t)-x_t\right\>+bt^q\left\<\frac{d}{dt}x_t,\dot{x}(t)\right\>
\end{align}

Observe further, that
$$\frac{1}{2} \| b(x(t)-x_t) + t^q \dot{x}(t) \|^2\le b^2\|x(t)-x_t\|^2+t^{2q}\|\dot{x}(t)\|^2,$$
hence, we have
\begin{align}\label{Lyapunovstrineq}
E(t)&\le t^{2q}(g_t(x(t))-g_t(x_t))+t^{2q}\|\dot{x}(t)\|^2+\frac{b(\a+b-qt^{q-1} )}{2}\|x(t)-x_t\|^2.
\end{align}
Consider now $r=\max(q,p-q)$ and $K>0$ that will be defined later.
Then, \eqref{strderiv1} and \eqref{Lyapunovstrineq} lead to

\begin{align}\label{strderivforlimit}
\dot{E}(t)+\frac{K}{t^r}E(t)&\le (2qt^{2q-1}-b t^q+Kt^{2q-r})(g_t(x(t))-g_t(x_t))- \frac{ap}{2t^{p-2q+1}}\left(\|x(t)\|^2-\|x_t\|^2\right)\\
\nonumber& +(b+qt^{q-1}-\a +Kt^{q-r})t^q\|\dot{x}(t)\|^2\\
\nonumber&+\frac{b(q(1-q)t^{q-2}-at^{q-p}+(\a+b-qt^{q-1})Kt^{-r} )}{2}\|x(t)-x_t\|^2
\\
\nonumber&+b(\a-qt^{q-1})\left\<\frac{d}{dt}x_t,x(t)-x_t\right\>+bt^q\left\<\frac{d}{dt}x_t,\dot{x}(t)\right\>,\mbox{ for all }t\ge t_1.
\end{align}

Now, it is obvious that for every $\b>0$ one has
$$\left\<\frac{d}{dt}x_t,x(t)-x_t\right\>\le \frac{\b}{2}t^{p-q}\left\|\frac{d}{dt}x_t\right\|^2+\frac{1}{2\b}t^{q-p}\|x(t)-x_t\|^2.$$
By using \eqref{fontos1} and the fact that there exists $t_2\ge t_1$ such that $b(\a-qt^{q-1})> 0$ for all $t\ge t_2$, we obtain that
\begin{equation}\label{xtderivest1}
b(\a-qt^{q-1})\left\<\frac{d}{dt}x_t,x(t)-x_t\right\>\le \frac{\b p^2 b(\a-qt^{q-1})}{2}t^{p-q-2}\left\|x_t\right\|^2+\frac{b(\a-qt^{q-1})}{2\b}t^{q-p}\|x(t)-x_t\|^2,
\end{equation}
for all $t\ge t_2$ and every $\b>0.$

Further, by using \eqref{fontos1} again, we obtain that for every $\g>0$ it holds
\begin{equation}\label{xtderivest2}
bt^q\left\<\frac{d}{dt}x_t,\dot{x}\right\>\le bt^q\left(\frac{ p^2 }{4\g t^2}\left\|x_t\right\|^2+\g\|\dot{x}(t)\|^2\right),\mbox{ for all }t\ge t_2.
\end{equation}

We recall that $\|x_t\|\le \|x^*\|$, hence by injecting \eqref{xtderivest1} and \eqref{xtderivest2} in \eqref{strderivforlimit} and neglecting the non-positive term $- \frac{ap}{2t^{p-2q+1}}\|x(t)\|^2$ we obtain
\begin{align}\label{strderivforlimit1}
\dot{E}(t)+\frac{K}{t^r}E(t)&\le (2qt^{2q-1}-b t^q+Kt^{2q-r})(g_t(x(t))-g_t(x_t))\\
\nonumber& +(b+qt^{q-1}-\a +Kt^{q-r}+b\g)t^q\|\dot{x}(t)\|^2\\
\nonumber&+\frac{b\left(q(1-q)t^{q-2}-at^{q-p}+(\a+b-qt^{q-1})Kt^{-r}+\frac{\a-qt^{q-1}}{\b}t^{q-p}\right)}{2}\|x(t)-x_t\|^2
\\
\nonumber&+ \left(\frac{ap}{2}t^{2q-p-1}+\frac{\b p^2 b(\a-qt^{q-1})}{2}t^{p-q-2}+\frac{b p^2}{4\g}t^{q-2}\right)\|x^*\|^2,\mbox{ for all }t\ge t_2\mbox{ and }\b,\g>0.
\end{align}

Since $0<b<\a$ one can consider $0<\g<\frac{\a}{b}-1$ and $\b>\frac{\a}{a}.$ Further, choose $K$ such that
$$0<K<\min\left(b,\a-b(1+\g),\frac{a\b-\a}{\b(\a+b)}\right).$$

Then, easily can be checked that there exists $t_3\ge t_2$ such that
\begin{equation}\label{forgronwal}
\dot{E}(t)+\frac{K}{t^r}E(t)\le  \left(\frac{ap}{2}t^{2q-p-1}+\frac{\b p^2 b(\a-qt^{q-1})}{2}t^{p-q-2}+\frac{b p^2}{4\g}t^{q-2}\right)\|x^*\|^2,\mbox{ for all }t\ge t_3.
\end{equation}

From now on  we treat the two cases (i) and (ii) separately.

{\bf (i)} If $p\ge\frac{3q+1}{2}$, then since by the hypotheses one has $0<q< 1$ we conclude that  there exists $C>0$ and $t_4\ge t_3$ such that
$$\left(\frac{ap}{2}t^{2q-p-1}+\frac{\b p^2 b(\a-qt^{q-1})}{2}t^{p-q-2}+\frac{b p^2}{4\g}t^{q-2}\right)\|x^*\|^2\le Ct^{p-q-2},\mbox{ for all }t\ge t_4.$$
Consequently, \eqref{forgronwal}becomes
\begin{equation}\label{forgronwal11}
\dot{E}(t)+\frac{K}{t^r}E(t)\le Ct^{p-q-2},\mbox{ for all }t\ge t_4.
\end{equation}

Since $0<q<1$ and $q+1>p\ge\frac{3q+1}{2}$ we get that $r=\max(q,p-q)=p-q<1$. By multiplying \eqref{forgronwal11} with
$e^{\frac{K}{1-r}t^{1-r}}$ we get

\begin{equation}\label{forgronwal2}
\frac{d}{dt}\left(e^{\frac{K}{1-r}t^{1-r}}E(t)\right)\le Ct^{p-q-2}e^{\frac{K}{1-r}t^{1-r}},\mbox{ for all }t\ge t_4.
\end{equation}
Observe that
$$\frac{d}{dt}\left(t^{p-q-2+r}e^{\frac{K}{1-r}t^{1-r}}\right)=((p-q-2+r)t^{p-q-3+r}+Kt^{p-q-2})e^{\frac{K}{1-r}t^{1-r}}$$
and since $p<q+1$ and $r<1$ we conclude that 
$$\frac{d}{dt}\left(t^{p-q-2+r}e^{\frac{K}{1-r}t^{1-r}}\right)\ge K t^{p-q-2}e^{\frac{K}{1-r}t^{1-r}},\mbox{ for all }t\ge t_4.$$
Hence, \eqref{forgronwal2} becomes
\begin{equation}\label{forgronwal3}
\frac{d}{dt}\left(e^{\frac{K}{1-r}t^{1-r}}E(t)\right)\le \frac{C}{K}\frac{d}{dt}\left(t^{p-q-2+r}e^{\frac{K}{1-r}t^{1-r}}\right),\mbox{ for all }t\ge t_4.
\end{equation}

Now, by integrating \eqref{forgronwal3} on an interval $[t_4,T],\,T>t_4$ we obtain
$$e^{\frac{K}{1-r}T^{1-r}}E(T)\le C_1T^{p-q-2+r}e^{\frac{K}{1-r}T^{1-r}}+C_2,$$
where $C_1=\frac{C}{K}$ and $C_2=e^{\frac{K}{1-r}t_4^{1-r}}E(t_4)-C_1t_4^{p-q-2+r}e^{\frac{K}{1-r}t_4^{1-r}}.$
In other words
$$E(t)\le C_1 t^{p-q-2+r}+\frac{C_2}{e^{\frac{K}{1-r}t^{1-r}}},\mbox{ for all }t\ge t_4.$$
Obviously $\frac{C_2}{e^{\frac{K}{1-r}t^{1-r}}}\le t^{p-q-2+r}$ if $t$ is big enough, hence there exists $t_5\ge t_4$ and $C_3>0$ such that
\begin{equation}\label{afinal}
E(t)\le C_3 t^{p-q-2+r},\mbox{ for all }t\ge t_5.
\end{equation}
Taking into account that $r=p-q$ and the definition of $E(t)$, from \eqref{afinal} we obtain at once that
\begin{equation}\label{eq2}
\|x(t)-x_t\|=O\left(\frac{1}{t^{q-p+1}}\right)\mbox{  as }t\to+\infty.
\end{equation}
Now, since $\lim_{t\to+\infty}x_t=x^*$ from \eqref{eq2} we get
$$\lim_{t\to+\infty}\|x(t)-x^*\|=0.$$
Further, \eqref{afinal} leads to
$$\|\dot{x}(t)\|=O\left(\frac{1}{t^{2q-p+1}}\right)\mbox{ as }t\to+\infty$$
and
$$g_t(x(t))-g_t(x_t)=O\left(\frac{1}{t^{4q-2p+2}}\right)\mbox{ as }t\to+\infty.$$
Now, if $4q-2p+2\ge p$, that is, $\frac{3q+1}{2}\le p\le \frac{4q+2}{3}$, from \eqref{fontos5} we obtain that
$$g(x(t))-\min g=O\left(\frac{1}{t^{p}}\right)\mbox{ as }t\to+\infty.$$

Conversely, if $4q-2p+2< p$, that is, $\frac{4q+2}{3}< p< q+1$, from \eqref{fontos5} we obtain
 $$g(x(t))-\min g=O\left(\frac{1}{t^{4q-2p+2}}\right)\mbox{ as }t\to+\infty.$$

{\bf (ii)}  Let us return now to \eqref{forgronwal}. If $p<\frac{3q+1}{2}$  we conclude that $q+1> p$, hence $t^{2q-p-1}\ge t^{q-2}$ and  $t^{2q-p-1}> t^{p-q-2}.$ From the latter relation we deduce that there exists $C>0$ and $t_4\ge t_3$ such that
$$\left(\frac{ap}{2}t^{2q-p-1}+\frac{\b p^2 b(\a-qt^{q-1})}{2}t^{p-q-2}+\frac{b p^2}{4\g}t^{q-2}\right)\|x^*\|^2\le Ct^{2q-p-1},\mbox{ for all }t\ge t_4.$$
Consequently, \eqref{forgronwal} becomes
\begin{equation}\label{forgronwal111}
\dot{E}(t)+\frac{K}{t^r}E(t)\le Ct^{2q-p-1},\mbox{ for all }t\ge t_4.
\end{equation}

Observe that $r<1$, hence proceeding analogously as in the previous case we obtain that there exists $C_1>0$ and $t_5\ge t_4$ such that
\begin{equation}\label{afinal1}
E(t)\le C_4 t^{2q-p-1+r},\mbox{ for all }t\ge t_5.
\end{equation}

Hence, we obtain at once that
$$\|\dot{x}(t)\|^2=O\left(\frac{1}{t^{p+1-r}}\right)\mbox{ as }t\to+\infty$$
and
$$g_t(x(t))-g_t(x_t)=O\left(\frac{1}{t^{p+1-r}}\right)\mbox{ as }t\to+\infty.$$
Now, from \eqref{fontos5} and the fact that $r<1$ we obtain
 $$g(x(t))-\min g=O\left(\frac{1}{t^p}\right)\mbox{ as }t\to+\infty.$$

Now, if $p\ge 2q$ we have $r=p-q$ hence $2q-p-1+r=q-1<0$ then, according to \eqref{Lyapunovstr} and \eqref{afinal} one has
$$\frac{b(\a-b-qt^{q-1} )}{2}\|x(t)-x_t\|^2\le C_4 t^{q-1},\mbox{ for all }t\ge t_6.$$
Consequently,
\begin{equation}\label{eq3}
\|x(t)-x_t\|=O\left(\frac{1}{t^{\frac{1-q}{2}}}\right)\mbox{  as }t\to+\infty.
\end{equation}

Conversely, if $p<2q$ then $r=q$, and since $E(t)\ge t^{2q}(g_t(x(t))-g_t(x_t))$ from \eqref{afinal1} and \eqref{fontos3} we get
$$\frac{a t^{2q-p}}{2}\|x(t)-x_t\|^2\le C_4 t^{2q-p-1+r},\mbox{ for all }t\ge t_5.$$
Hence,
\begin{equation}\label{eq4}
\|x(t)-x_t\|=O\left(\frac{1}{t^{\frac{1-q}{2}}}\right)\mbox{  as }t\to+\infty.
\end{equation}
Now, since $\lim_{t\to+\infty}x_t=x^*$ from \eqref{eq3} and \eqref{eq4} we get
$$\lim_{t\to+\infty}\|x(t)-x^*\|=0.$$
\end{proof}

\begin{remark} Observe that our analysis presented in the proof of Theorem \ref{StrongConvergenceResultLyapunov} also works for the cases $p=q+1$ and $q=1,\,p\le 2$, unfortunately one can not obtain an improvement of the results already presented at Theorem \ref{Oestimates} and Remark \ref{r3}.
\end{remark}

\section{Numerical experiments}
In this section we consider two numerical experiments for the trajectories generated by the dynamical system \eqref{DynSys} for a convex but not strongly convex objective function
$$g:\R^2\to\R,\,\,\,g(x,y)=(mx+ny)^2\mbox{ where }m,n\in\R\setminus\{0\}.$$
Observe that $\argmin g=\left\{\left(x,-\frac{m}{n}x\right):x\in\R\right\}$ and $\min g=0$. Obviously, the minimizer of minimal norm of $g$ is $x^*=(0,0).$
In the following numerical experiments we consider the starting points $x(1)=(1,1),\,\dot{x}(1)=(-1,-1)$ and the continuous time dynamical system \eqref{DynSys} is solved numerically with the ode45 adaptive method in MATLAB on the interval $[1, 100]$.

In our first experiment we take  $m=5$ and $n=1$, values for which the function $g$ is well conditioned. Further, we consider $\a=3.5$, $a=1$ and $p=1.2$ and we study the evolution of the two errors $\|x(t)-x^*\|$ and $g(x(t))-\min g$, for a trajectory $x(t)$ generated by the dynamical system \eqref{DynSys}, with different values of $q$.  The results are depicted on Figure 1, where the $y$ axis is endowed with a logarithmic scale.
\begin{figure}[hbt!]
\begin{subfigure}{.49\textwidth}
  \centering
  \includegraphics[width=.99\linewidth]{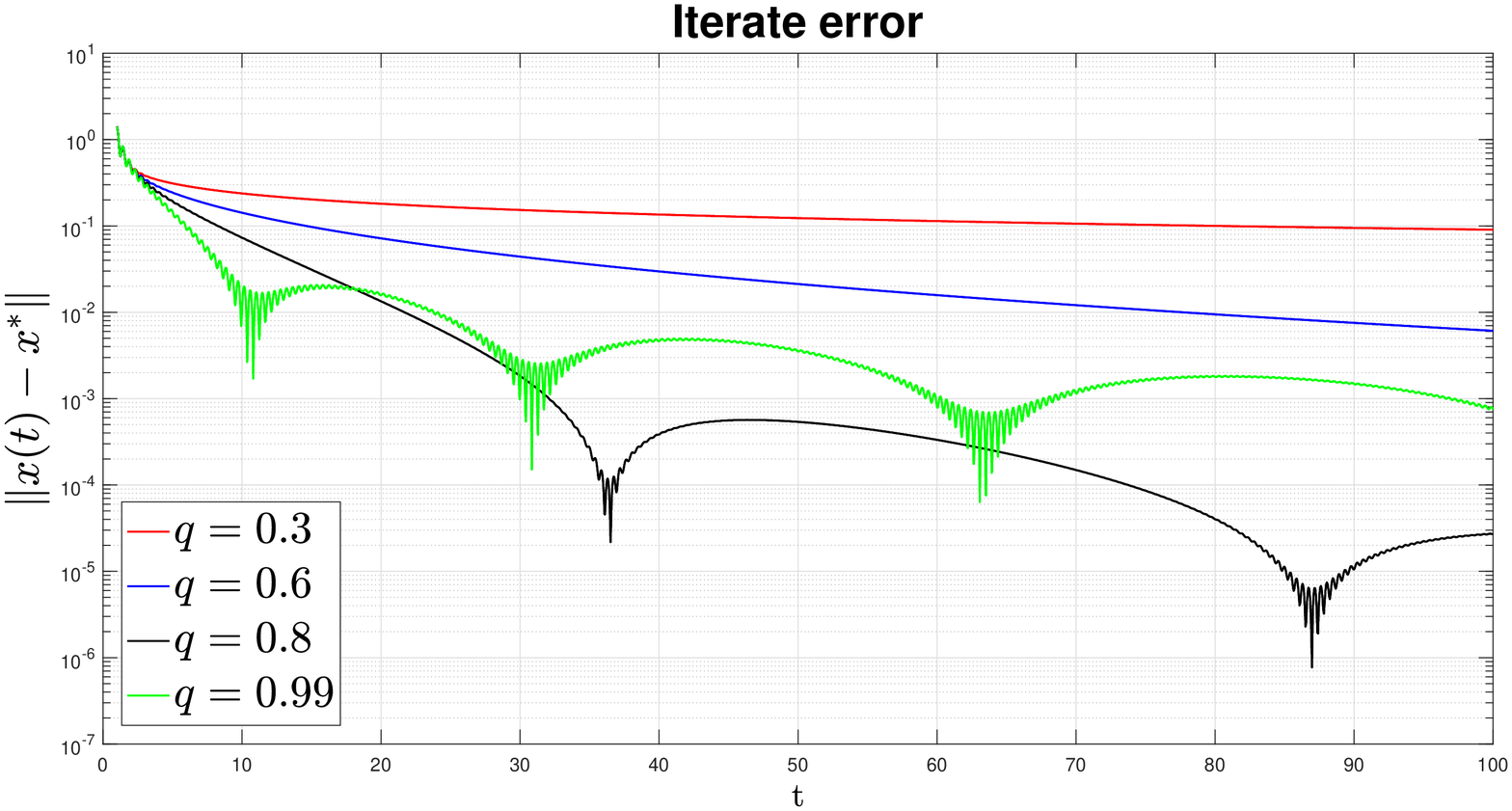}
\end{subfigure}
\begin{subfigure}{.49\textwidth}
  \centering
  \includegraphics[width=.99\linewidth]{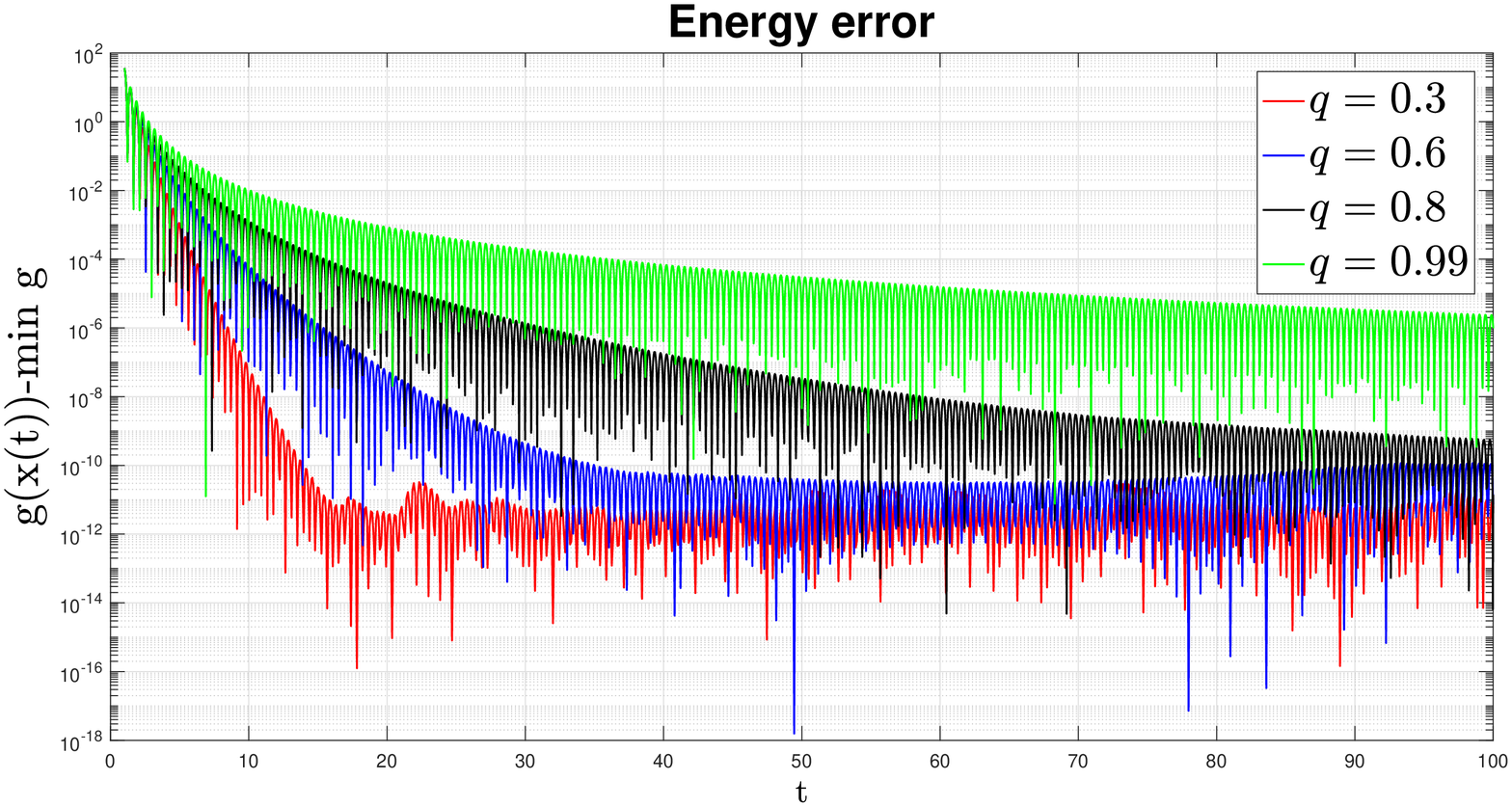}
\end{subfigure}
 \caption{Error analysis with different damping parameters in the dynamical system \eqref{DynSys}  for a well conditioned convex objective function.}
\end{figure}\label{fig11}

Note that the according to our numerical experiment for $p=1.2$ the best  convergence  result for the iterate error $\|x(t)-x^*\|$ is achieved for $q=0.99$, meanwhile the best best  convergence  result for the energy error $g(x(t))-\min g$ is attained for $q=0.3.$

In our second experiment we fix $q=0.7$ and we take different values of $p.$ The other parameters remain unchanged. The results are depicted on Figure 2, where the $y$ axis is endowed with a logarithmic scale.

\begin{figure}[hbt!]
\begin{subfigure}{.49\textwidth}
  \centering
  \includegraphics[width=.99\linewidth]{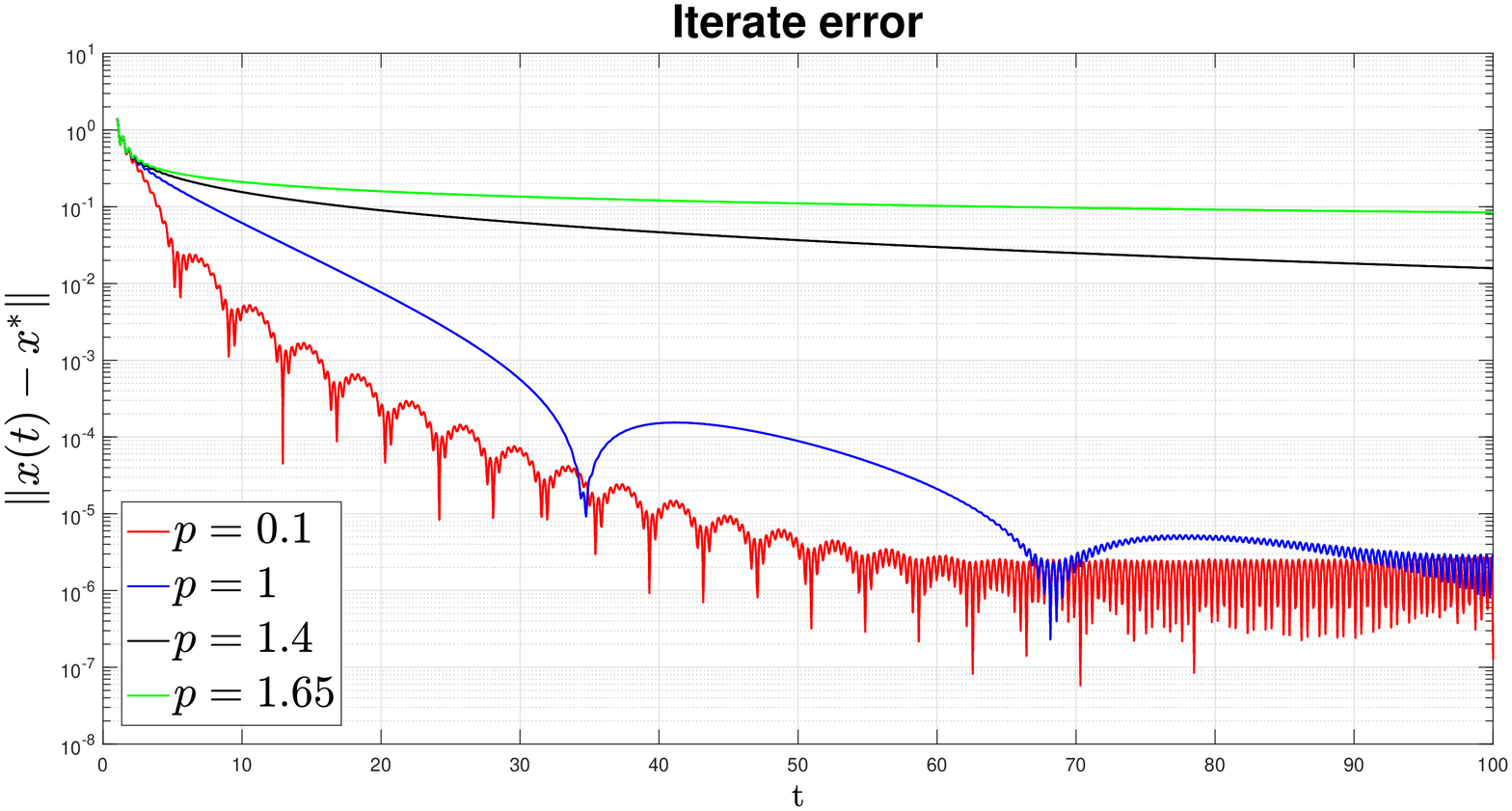}
\end{subfigure}
\begin{subfigure}{.49\textwidth}
  \centering
  \includegraphics[width=.99\linewidth]{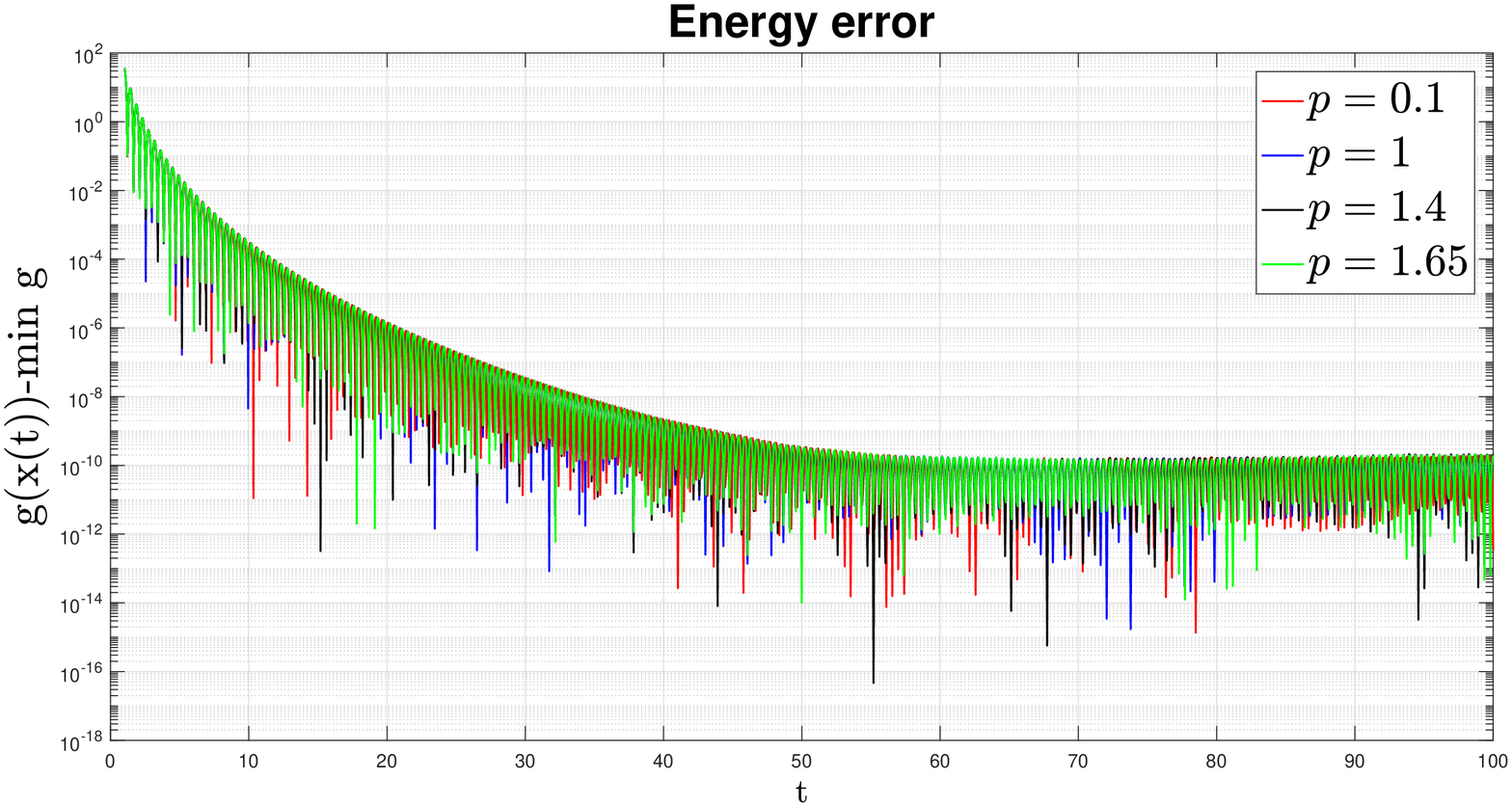}
\end{subfigure}
 \caption{Error analysis with different Tikhonov regularization  parameters in the dynamical system \eqref{DynSys}  for a well conditioned convex objective function.}
\end{figure}\label{fig21}
We conclude that also for fixed damping, one obtains the best convergence behaviour for the iterate error $\|x(t)-x^*\|$ for the cases when $p$ is considerably less than $2q.$
Further, one can easily observe that the energy error $g(x(t))-\min g$ is not very sensitive for the changes of the Tikhonov regularization parameter.

\appendix
\section{Appendix}\label{app}

\subsection{Existence and uniqueness for the Cauchy problem}
Let us first show that the solution for \eqref{DynSys} is well posed.

\begin{theorem}\label{Cauchy-weel-posed}
Given $(x_0, v_0) \in \mathcal{H} \times \mathcal{H}$, there exists a unique global classical solution $x : [t_0, +\infty[ \to \mathcal{H}$ of the dynamical system \eqref{DynSys}.
\end{theorem}

\begin{proof} The proof  relies on the combination of the Cauchy-Lipschitz theorem with energy estimates.
First consider the Hamitonian formulation of  \eqref{DynSys} as the first order system
\begin{align}\label{DynSys-2}
\begin{cases}
\dot{x}(t) - y(t) =0 \vspace{1mm} \\
\dot{y}(t) + \frac{\a}{t^q} y(t) + \nabla g\left(x(t) \right)+\frac{a}{t^p}x(t)=0 \vspace{1mm}\\
x(t_0) = u_0, \,
y(t_0) = v_0.
\end{cases}
\end{align}
Taking into account the hypotheses and by applying the Cauchy-Lipschitz theorem in the locally Lipschitz case, we obtain the existence and uniqueness of a local solution.
Then, in order to pass from a local solution to a global solution, we rely on the  energy estimate obtained by taking the scalar product of \eqref{DynSys} with $\dot{x}(t)$. It gives
$$
\frac{d}{dt} \left( \frac12 \| \dot{x}(t)\|^2 + g(x(t)) +  \frac{a}{2 t^p} \|x(t)\|^2  \right)
+ \frac{\a}{t^q} \| \dot{x}(t)\|^2 + \frac{ap}{2t^{p+1}} \|x(t)\|^2 =0.
$$
Obviously, the energy function $t \mapsto W(t)$ is decreasing where
$$
W(t):=  \frac12 \| \dot{x}(t)\|^2 + g(x(t)) +  \frac{a}{2 t^p} \|x(t)\|^2.
$$
The end of the proof follows a standard argument. Take a maximal solution defined on an interval $[t_0, T_{\max}[$. If $T_{\max}$ is infinite, the
proof is over. Otherwise, if $T_{\max}$ is finite, according to the above energy estimate, we have that $\| \dot{x}(t)\|$ remains bounded.
Let $\|\dot{x}_{\infty}\|=\sup_{t\in [t_0, T_{\max})}\|\dot{x}(t)\|.$ Since $\|x(t)-x(t')\|\le\|\dot{x}_{\infty}\||t-t'|$, we get that $\lim_{t\to T_{\max}}x(t):= x_{\infty}\in \mathcal{H}$. By \eqref{DynSys} the map $\ddot{x}$ is also bounded on the interval $[t_0, T_{\max})$ and under the same argument as before
$\lim_{t\to T_{\max}}\dot{x}(t):= x_{\infty}$ exists. Applying the local existence theorem with initial data $(x_{\infty},\dot{x}_{\infty})$, we can extend the maximal solution to a strictly larger interval, a clear contradiction. Hence $T_{\max}=+\infty,$ which completes the proof.
\end{proof}

\section{Auxiliary results}
In this appendix, we collect some lemmas and technical results which we will use in the analysis of the dynamical system \eqref{DynSys}.

The following statement is the continuous counterpart of a convergence result
of quasi-Fej\'er monotone sequences. For its proofs we refer to \cite[Lemma 5.1]{abbas-att-sv}.

\begin{lemma}\label{fejer-cont1} Suppose that $F:[t_0,+\infty)\rightarrow\R$ is locally absolutely continuous and bounded from below and that
there exists $G\in L^1([t_0,+\infty))$ such that
$$\frac{d}{dt}F(t)\leq G(t)$$
for almost every $t \in [t_0,+\infty)$. Then there exists $\lim_{t\To +\infty} F(t)\in\R$.
\end{lemma}

The continuous version of  the Opial Lemma (see \cite{att-c-p-r-math-pr2018}) is the main tool for proving weak convergence for the generated trajectory.

\begin{lemma}\label{Opial}
Let $S \subseteq \mathcal{H}$ be  a nonempty set and $x : [t_0, +\infty) \to H$ a given map such that:
\begin{align*}
& (i) \quad \mbox{for every }z \in S \ \mbox{the limit} \ \lim\limits_{t \To +\infty} \| x(t) - z \|  \ \mbox{exists};\\
& (ii) \quad \text{every weak sequential limit point of } x(t) \text{ belongs to the set }S.
\end{align*}
Then the trajectory $x(t)$ converges weakly to an element in $S$ as $t \to + \infty$.
\end{lemma}

Inspired from \cite{AP-max} Lemma A.6, we have the following result.

\begin{lemma}\label{A1} Let $t_0 > 0$, and let $w : [t_0,+\infty[\to\R$ be a continuously differentiable function which
is bounded from below. Consider $p,q,k: [t_0,+\infty[\to\R$  nonnegative functions and assume that the function $\frac{1}{\exp\left(\int_{t_0}^t \frac{q(s)}{p(s)}ds\right)} \in L^1(t_0,+\infty[$ and
$\left(\int_{t}^{+\infty}\frac{dT}{\exp\left(\int_{t_0}^T \frac{q(s)}{p(s)}ds\right)}\right)\frac{\exp\left(\int_{t_0}^t\frac{q(s)}{p(s)}ds\right)}{p(t)}k(t)\in L^1(t_0,+\infty[.$
Assume further, that
\begin{equation}\label{abst}
p(t)\ddot{w}(t) + q(t)\dot{w}(t)\le k(t)
\end{equation}
for some $\a > 0$, almost every $t > t_0$. Then,
the positive part $[\dot{w}]_+$ of $\dot{w}$ belongs to $L^1(t_0,+\infty[$, and $\lim_{t\to+\infty} w(t)$ exists.
\end{lemma}
\begin{proof}
By multiplying \eqref{abst} with $\frac{1}{p(t)}\exp\left(\int_{t_0}^t \frac{q(s)}{p(s)}ds\right)$ we get

$$\frac{d}{dt}\left[\exp\left(\int_{t_0}^t\frac{q(s)}{p(s)}ds\right)\dot{w}(t)\right]\le \frac{1}{p(t)}\exp\left(\int_{t_0}^t\frac{q(s)}{p(s)}ds\right)k(t).$$
By integrating the above relation on an interval $[t_0, T]$ we obtain
\begin{equation}\label{abst1}
\exp\left(\int_{t_0}^T \frac{q(s)}{p(s)}ds\right)\dot{w}(T)\le \dot{w}(t_0)+\int_{t_0}^T \frac{1}{p(t)}\exp\left(\int_{t_0}^t\frac{q(s)}{p(s)}ds\right)k(t)dt.
\end{equation}

Consequently,

\begin{equation}\label{abst2}
[\dot{w}]_+(T)\le \frac{|\dot{w}(t_0)|}{\exp\left(\int_{t_0}^T \frac{q(s)}{p(s)}ds\right)}+\frac{\int_{t_0}^T \frac{1}{p(t)}\exp\left(\int_{t_0}^t\frac{q(s)}{p(s)}ds\right)k(t)dt}{\exp\left(\int_{t_0}^T \frac{q(s)}{p(s)}ds\right)},
\end{equation}
hence
\begin{equation}\label{abst3}
\int_{t_0}^{+\infty}[\dot{w}]_+(T)dT\le \int_{t_0}^{+\infty}\frac{|\dot{w}(t_0)|}{\exp\left(\int_{t_0}^T \frac{q(s)}{p(s)}ds\right)}dT+\int_{t_0}^{+\infty}\frac{\int_{t_0}^T \frac{1}{p(t)}\exp\left(\int_{t_0}^t\frac{q(s)}{p(s)}ds\right)k(t)dt}{\exp\left(\int_{t_0}^T \frac{q(s)}{p(s)}ds\right)}dT.
\end{equation}
Now, from the hypotheses we have
$$\int_{t_0}^{+\infty}\frac{|\dot{w}(t_0)|}{\exp\left(\int_{t_0}^T \frac{q(s)}{p(s)} ds\right)}dT<+\infty.$$
According to Fubini's theorem
$$\int_{t_0}^{+\infty}\frac{\int_{t_0}^T \frac{1}{p(t)}\exp\left(\int_{t_0}^t\frac{q(s)}{p(s)}ds\right)k(t)dt}{\exp\left(\int_{t_0}^T \frac{q(s)}{p(s)}ds\right)}dT=\int_{t_0}^{+\infty}\left(\int_{t}^{+\infty}\frac{dT}{\exp\left(\int_{t_0}^T \frac{q(s)}{p(s)}ds\right)}\right)\frac{\exp\left(\int_{t_0}^t\frac{q(s)}{p(s)}ds\right)}{p(t)}k(t)dt.$$
But, according to the hypotheses
$$\int_{t_0}^{+\infty}\left(\int_{t}^{+\infty}\frac{dT}{\exp\left(\int_{t_0}^T \frac{q(s)}{p(s)}ds\right)}\right)\frac{\exp\left(\int_{t_0}^t\frac{q(s)}{p(s)}ds\right)}{p(t)}k(t)dt<+\infty, $$
consequently
$$\int_{t_0}^{+\infty}[\dot{w}]_+(T)dT<+\infty.$$
This implies that $\lim_{t\to+\infty} w(t)$ exists.
\end{proof}

\end{document}